\documentclass[a4paper,12pt]{amsart}
\usepackage{amsmath,amsfonts,amssymb,amsthm,amscd}
\usepackage{latexsym,graphicx}
\usepackage[matrix,arrow,ps,color,line,curve,frame]{xy}
\usepackage[usenames,dvipsnames]{color}

\usepackage{xcolor}%TMP
\newcommand{\half}{\frac{1}{2}}

\renewcommand{\[}{\begin{equation}}
\renewcommand{\]}{\end{equation}}

\newtheorem{proposition}{\sc Proposition}[section]

\newtheorem{theorem}[proposition]{\sc Theorem}
\newtheorem{conjecture}[proposition]{\sc Conjecture}
\newtheorem{definition}[proposition]{\sc Definition}
\theoremstyle{definition}

\theoremstyle{remark}

\newcommand{\id}{\operatorname{id}}

\newcommand{\ra}{\rightarrow}

\newcommand{\ot}{\otimes}
\renewcommand{\phi}{\varphi}
\renewcommand{\epsilon}{\varepsilon}

\newcommand{\cO}{\mathcal{O}}

\def\C{{\mathbb C}}

\def\z2{{\mathbb Z}/2{\mathbb Z}}

\def\id{{\rm id}}

\def\pr{{\rm pr}}

\def\half{{\mbox{$\frac{1}{2}$}}}

\newcommand{\lra}{\longrightarrow}

\usepackage[T1]{fontenc}
\usepackage{stmaryrd}

\begin{document}
\baselineskip=16pt

\author{Paul F. Baum}
\address{Mathematics Department, McAllister Building,  
The Pennsylvania State University,
University Park, PA  16802, USA\\
Instytut Matematyczny, Polska Akademia Nauk, ul.~\'Sniadeckich 8, Warszawa, 00-656 Poland}
\email{baum@math.psu.edu}

\author{Ludwik D\k abrowski}
\address{SISSA (Scuola Internazionale Superiore di Studi Avanzati)\\
Via Bonomea 265, 34136 Trieste, Italy
}
\email{dabrow@sissa.it}

\author{Piotr M.~Hajac}
\address{Institytut Matematyczny, Polska Akademia Nauk\\
ul.\ \'Sniadeckich 8, 00-656 Warszawa, Poland\\
}
\email{pmh@impan.pl}

\title[Noncommutative Borsuk-Ulam-type conjectures]{\large Noncommutative Borsuk-Ulam-type conjectures}
%\vspace*{-15mm}
\begin{abstract}
Within the framework of free actions of compact quantum groups on unital\linebreak \mbox{C*-algebras}, we propose two conjectures. 
The first one states
that, if $H$ is the C*-algebra of a compact quantum group coacting freely on a unital C*-algebra $A$,  
then there is no equivariant $*$-homomorphism from $A$ to the join C*-algebra $A*H$. 
For $A$ being the C*-algebra
of continuous functions on a sphere with the antipodal coaction of 
the C*-algebra of funtions on $\mathbb{Z}/2\mathbb{Z}$,
we recover the celebrated  Borsuk-Ulam theorem. 
The second conjecture states that there is no equivariant $*$-homomorphism 
from $H$ to the join C*-algebra $A*H$.  
We show how to prove the conjecture in the special case $A=C(SU_q(2))=H$, which
is tantamount to showing the non-trivializability of Pflaum's quantum instanton fibration built 
from $SU_q(2)$.
\end{abstract}
\maketitle
%\tableofcontents

\section{Introduction}
\noindent
The Borsuk-Ulam theorem is a fundamental theorem of topology with an enormous amount
of corollaries and generalizations~\cite{mj03}. Having a noncommutative generalization would provide
further evidence for its fundamental nature.

\begin{theorem}[\cite{b-k33}]
Let $n$ be a positive natural number.
If $f\colon S^n\to\mathbb{R}^n$ is continuous, then there exists a pair $(p,-p)$ of antipodal points on
$S^n$ such that $f(p)=f(-p)$.
\end{theorem}
\noindent
Assuming that both temperature and pressure are continuous functions, we can conclude that there are always 
two antipodal points on Earth with exactly the same pressure and temperature.

The logical negation of the theorem yields:
There exists a continuous map $f\colon S^n\to\mathbb{R}^n$ such that for all pairs 
$(p,-p)$ of antipodal points on
$S^n$ we have $f(p)\neq f(-p)$.
For the antipodal action of $\mathbb{Z}/2\mathbb{Z}$ on $S^n$ and $\mathbb{R}^n$, the latter
statement is equivalent to:
%{Equivalent negation}
There  exists a $\mathbb{Z}/2\mathbb{Z}$-equivariant 
continuous  map $\widetilde{f}\colon S^n\to S^{n-1}$. 

Indeed, if $f\colon S^n\to\mathbb{R}^n$ is a continuous map with $f(p)\neq f(-p)$ , 
then the formula
\[
\widetilde{f}(p):=\frac{f(p)-f(-p)}{\|f(p)-f(-p)\|}
\]
defines a continuous $\mathbb{Z}/2\mathbb{Z}$-equivariant map from $S^n$ to $S^{n-1}$.
Also, composing any such a map with the inclusion map $S^{n-1}\subset\mathbb{R}^n$ gives
a nowhere vanishing continuous map $f\colon S^n\to\mathbb{R}^n$ with 
$f(-p)=-f(p)\neq f(p)$.
Consequently, the Borsuk-Ulam Theorem is equivalent to:
\begin{theorem}[equivariant formulation]\label{classicalequi}
Let $n$ be a positive natural number.
There  does \emph{not} exist a $\mathbb{Z}/2\mathbb{Z}$-equivariant 
continuous  map $\widetilde{f}\colon S^n\to S^{n-1}$. 
\end{theorem}

\subsection{Classical equivariant join construction}

\noindent
Let $I=[0,1]$ be the closed unit interval and let $X$ be a topological space. 
The \emph{unreduced suspension} $\Sigma X$
of $X$ is the quotient of $I\times X$ by the equivalence relation $R_S$ generated by
\begin{align}
&(0,x)\sim (0,x'),\qquad (1,x)\sim (1,x').
\end{align}
Now take another topological space $Y$ and, on the space $I\times X\times Y$, consider the 
 equivalence relation  $R_J$ given by
\begin{equation}
(0,x,y)\sim (0,x',y),\qquad (1,x,y)\sim (1,x,y').
\end{equation}
The quotient space 
$X * Y:=(I\times X\times Y)/R_J$ 
is called the \emph{join} of $X$ and $Y$. It resembles the unreduced suspension of $X\times Y$, 
but with only $X$ collapsed at 0, and only $Y$ 
collapsed at~1. 
\vspace*{-0mm}\begin{figure}[h]
\[
\includegraphics[width=105mm]{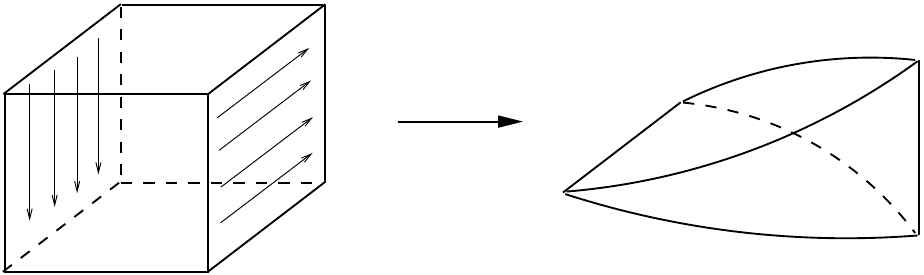}
\nonumber
\]
\end{figure}\vspace*{-0mm}
In particular, if $Y$ is a one-point space, the join $X*Y$ is  the {\em cone} $CX$ of~$X$.
If $Y$ is a two-point space with discrete topology, then the join $X*Y$ is the unreduced suspension $\Sigma X$ 
of~$X$.

If $G$ is a topological group acting  continuously on  $X$ and $Y$ from the right,
then the diagonal right \mbox{$G$-action} on $X\times Y$ induces a  continuous action on the join $X*Y$.
 Indeed,
the diagonal action of $G$ on $I\times X\times Y$ factorizes to 
the quotient, so that the formula
\[
([(t,x,y)],g)\longmapsto [(t,xg,yg)]
\]
makes $X*Y$ a right $G$-space. It is immediate that this continuous action is free 
if the $G$-actions on $X$ and $Y$ are free. 

If $Y=G$ with the right action assumed to be the group multiplication, 
we can construct the join $G$-space $X * Y$ in a  different manner: 
at 0 we collapse $X \times G$ to $G$ as before,
and at 1 we collapse $X \times G$  to $(X\times G)/R_D$ 
instead of~$X$. Here $R_D$ is the equivalence relation generated by
$(x,h)\sim (x',h')$, where $xh=x'h'$.
More precisely, let
 $R_J'$ be the  equivalence relation on $I \times X\times G$ generated by 
\[\label{5}
(0,x,h)\sim (0,x',h)\quad \text{and}\quad (1,x,h)\sim (1,x',h'), \; \text{where}\; xh=x'h'.
\]
The formula $[(t,x,h)]k:=[(t,x,hk)]$ defines a continuous right $G$-action on $(I \times X\times G)/R_J'$.
One can also easily check that the formula
\[\label{6}
X*G\ni [(t,x,h)] \longmapsto [(t,xh^{-1},h)] \in (I\times X\times G)/R_J'
\]
yields a $G$-equivariant homeomorphism. 

If we further specify also $X=G$ with the right action assumed to be the group multiplication,
  then the $G$-action on $X*Y=G * G$ is automatically free.  Furthermore, since the  action of $G$ on $X*G$ is 
free whenever it is free on $X$,
we conclude that the natural action on the iterated join of $G$ with itself is also free.
For instance, for $G=\z2$ we obtain a $\z2$-equivariant identification
$(\z2)^{*(n+1)}\cong S^{n}$,
where $S^{n}$ is the $n$-dimensional sphere with the antipodal action of $\z2$.

\subsection{Join of unital C*-algebras with free compact quantum group actions}

\begin{definition}[\cite{dhh}]\label{join}
Let $A_1$ and $A_2$  be unital C*-algebras. We call the unital C*-algebra
\begin{equation}
A_1 * A_2 := 
\left\{x\in C([0,1]) \underset{\mathrm{min}}{\otimes} A_1 \underset{\mathrm{min}}{\otimes} A_2 \,
\Big|\, (\mathrm{ev}_0\otimes \id)(x) \in  A_2,\,(\mathrm{ev}_1\otimes \id)(x) \in  A_1\right\}
\end{equation}
the \emph{join} C*-algebra of $A_1$ and~$A_2$. Here $\otimes_{\mathrm{min}}$ stands for the spatial 
(minimal) tensor product, and
 $\mathrm{ev}_r\colon C([0,1])\ni f\mapsto f(r)\in\mathbb{C}$
is the evaluation map at $r$ on the C*-algebra of continuous functions on the unit interval.
\end{definition}
\noindent
Note that, due to the fact that minimal tensor products preserve injections 
(e.g., see \cite[Proposition~4.22]{t-m79} or 
\cite[Section~1.3]{w-s94}), 
for any unital C*-algebras $A_1$ and $A_2$ the natural maps
\[
A_1\ni a\longmapsto a\ot 1\in A_1 \underset{\mathrm{min}}{\otimes} A_2\quad \text{and}
\quad  A_2\ni a\longmapsto 1\ot a\in A_1 \underset{\mathrm{min}}{\otimes} A_2
\]
are injective. Observe also that,
if $A_1:=C(X)$ and $A_2 := C(Y)$ are the C*-algebras of continuous functions on compact Hausdorff spaces 
$X$ and $Y$ respectively,
then 
\[
A_1 * A_2 = C(X * Y).
\]

Next, let $A$ be a unital $C^*$-algebra and
$\delta:A\rightarrow A\otimes_{\mathrm{min}}H$ an
injective  unital $*$-homomorphism. We call $\delta$ a \emph{coaction}  
of $H$ on $A$ (or an action of the compact quantum group $(H,\Delta)$ on $A$) iff
\begin{enumerate}
\item
$(\delta\otimes\mathrm{id})\circ\delta=
(\mathrm{id}\otimes\Delta)\circ\delta$
(coassociativity),
\item
$\{\delta(a)(1\otimes h)\;|\;a\in A,\,h\in H\}^{\mathrm{cls}}=
A\underset{\mathrm{min}}{\otimes}H$ (counitality).
\end{enumerate}
Here ``cls'' stands for ``closed linear span''. A coaction $\delta$ is called \emph{free} \cite{e-da00} iff
\[
\{(x\otimes 1)\delta(y)\;|\;x,y\in A\}^{\mathrm{cls}}=
A\underset{\mathrm{min}}{\otimes}H.
\] 

Since a diagonal coaction is not in general an algebra homomorphism,
to obtain an equivariant version of our noncommutative join construction, 
we need to modify  Definition~\ref{gjoin}  
in the spirit of \eqref{5}--\eqref{6}. 
\begin{definition}[cf.~\cite{dhh}]\label{gjoin}
Let $(H,\Delta)$ be a compact quantum group acting on a unital $C^*$-algebra~$A$
 via $\delta:A\rightarrow A\otimes_{\mathrm{min}}H$.
We call the unital $C^*$-algebra
\begin{equation}
A \underset{\delta}{*} H := 
\left\{f\in C\big([0,1],A \underset{\mathrm{min}}{\otimes} H\big) \,
\Big|\, f(0) \in \mathbb{C}\otimes H,\; f(1)\in\delta(A)\right\}
\end{equation}
the \emph{equivariant noncommutative join} of $A$ and $H$. 
\end{definition}
\begin{theorem}\mbox{\sc (\cite[Lemma~5.5 and Corollary~5.6]{bdh})}\ 
Let $(H,\Delta)$ be a compact quantum group, and $A$ be a unital C*-algebra equipped with a free coaction 
$\delta\colon A\to A \otimes_{\mathrm{min}}H$.\\
The $*$-homomorphism  
$$
\mathrm{id}\!\otimes\!\Delta\colon\; C([0,1],A)\! \underset{\mathrm{min}}{\otimes} \!H\;\longrightarrow\;
C([0,1],A)\! \underset{\mathrm{min}}{\otimes} \!H\! \underset{\mathrm{min}}{\otimes} \!H
$$ 
restricts to
\begin{equation}\label{deltadelta}
\delta_\Delta\colon\; A\underset{\delta}{*} H\;\longrightarrow\;
(A\underset{\delta}{*} H) \underset{\mathrm{min}}{\otimes} H.
\end{equation}
Moreover, the thus defined action of the compact quantum group
$(H,\Delta)$  on the join C*-algebra $A*_{\delta}H$ is  \emph{free}.
\end{theorem}
Given a compact quantum group $(H,\Delta)$, we denote by
 $\cO (H)$ its dense
Hopf $*$-subalgebra spanned by the matrix coefficients of
irreducible unitary corepresentations (see Worono\-wicz's Peter-Weyl theory of compact quantum
 groups~\cite{w-sl98}).
Moreover, denoting by $\otimes$  the purely algebraic tensor product over the field
	$\mathbb{C}$ of complex numbers, we define the
\emph{Peter-Weyl subalgebra} of $A$ \cite{bdh} as
\[
\mathcal{P}_H(A):=\{\,a\in A\,| \,\delta(a)\in A\otimes\cO (H)\,\}.
\]
Using the coassociativity of $\delta$, one can  check that $\mathcal{P}_H(A)$ is a right 
$\cO (H)$-comodule algebra. In particular, 
$\mathcal{P}_H(H)=\mathcal{O}(H)$.
The assignment $A\mapsto\mathcal{P}_H(A)$ is functorial with respect to
equivariant unital $*$-homomorphisms and comodule algebra maps. We call it the
\emph{Peter-Weyl functor}.

\subsection{Pullback structure of equivariant joins}
Let $(H,\Delta)$ be a compact quantum group, $\mathcal{O}(H)$ be its 
dense Hopf $*$-subalgebra
spanned by the matrix coefficients of irreducible unitary corepresentations.
Set 
\begin{eqnarray}
&& \mathcal{P}_1:=\left\{f \in C([0,\half],H) \otimes \mathcal{O}(H)\,|\,
	f(0) \in \mathbb{C} \otimes \mathcal{O}(H)
	\right\},\nonumber\\ 
&& \mathcal{P}_2:=\left\{f \in C([\half,1],H) \otimes \mathcal{O}(H)\,|\,
	f(1) \in \Delta(\mathcal{O}(H)) \right\},\nonumber\\
&& \mathcal{P}_{12}:= H \otimes\mathcal{O}(H).
\end{eqnarray} 
Here we identify elements of 
$C(I, H) \otimes \mathcal{O}(H)$ with functions 
$I \rightarrow  H \otimes \mathcal{O}(H)$. The just defined spaces
are $\mathcal{O}(H)$-comodule algebras with respect to the coaction
$\mathrm{id}\otimes \Delta$.
The coaction-invariant subalgebras $\mathcal{P}_{i}^{co\,\mathcal{O}(H)}$,  $i=1,2$,
 can be identified respectively with the ``left'' and ``right'' cone of~$H$:
\begin{eqnarray}
&& C_1H:=\left\{f \in C([0,\half], H)\,|\,f(0) \in \mathbb{C}\right\},
		  \nonumber\\ 
&& C_2H:=\left\{f \in C([\half,1], H)\,|\,f(1) \in \mathbb{C}\right\}.
		  \end{eqnarray}
We also see that $\mathcal{P}_{12}^{co\,\mathcal{O}(H)}= H$ and
$\mathcal{P}_1 \cong C_1H \otimes \mathcal{O}(H)$, 
$\mathcal{P}_2 \cong C_2H \# \mathcal{O}(H) $. Here $\mathcal{O}(H)$ acts on $C_2H$ via
the adjoint action
\[
(a \triangleright f)(t)=a_{(1)}f(t)S(a_{(2)}),\quad
a \in H,\quad f \in C_1H,\quad t \in [0,1].
\]

Now let $\pi^P_i\colon \mathcal{P}_i\to \mathcal{P}_{12}$, $i=1,2$, be the evaluation maps  at $\half$.
Then 
\[\label{P_is_fiber_p} 
\mathcal{P}_1 \underset{\mathcal{P}_{12}}{\times}  \mathcal{P}_2
:=\{ (p_1,p_2)\in \mathcal{P}_1 \times \mathcal{P}_2 ~|~ \pi^P_1(p_1)= \pi^P_2(p_2)\}
\]
is the pullback comodule algebra for the diagram
\begin{equation}\label{P_is_fiber_product} 
\mbox{$\xymatrix@=5mm{& & 
\mathcal{P}_1 \underset{\mathcal{P}_{12}}{\times} \mathcal{P}_2 
\ar[lld]_{\pr_1}
\ar[rrd]^{\pr_2}& &\\
\mathcal{P}_1 \ar[drr]_{\pi^P_1}& & & &\mathcal{P}_2\,.\ar[dll]^{\pi^P_2}\\
&&\mathcal{P}_{12}&&}$}
\end{equation}
Since the aforementioned evaluation maps are surjective, we conclude from \cite{hkmz11} that all four 
comodule algebras in this diagram are principal.
Furthermore, one can check \cite{ddhw}  that the Peter-Weyl comodule algebra of the join C*-algebra 
$H\!*_\Delta\! H$ coincides with the above pullback
comodule algebra:
\[
\mathcal{P}_H(H\!*_\Delta\! H)\cong 
\mathcal{P}_1 \underset{\mathcal{P}_{12}}{\times}  \mathcal{P}_2\ .
\]

Herein, the coaction-invariant subalgebras are C*-algebras. They assemble into the pullback diagram along the evaluation maps at $\half$:
\begin{equation}\label{B_is_fiber_product} 
\mbox{$\xymatrix@=5mm{& & \Sigma H
\ar[lld]_{\pr_1}
\ar[rrd]^{\pr_2}& &\\
C_1H \ar[drr]_{\pi_1}& & & &C_2H\,, \ar[dll]^{\pi_2}\\
&& H&&}$}
\end{equation}
where $\Sigma H$ is the unreduced suspension C*-algebra of~$H$.

We end this section by unravelling the Milnor construction \cite{m-j71} 
for the specific case of the above pullback of unital C*-algebras.
Given a finite-dimensional complex vector space $V$ and an isomorphism of left $H$-modules
$\chi:H\ot V\ra H\ot V$,
we construct the finitely generated projective left $\Sigma H$-module 
$M\left(C_1H\ot V ,\, C_2H\ot V,  \chi\right)$ 
(see~\cite{dhhmw12})  as the 
pullback of the free left $C_1H$-module $C_1H\ot V$ and the free left $C_2H$-module $C_2H\ot V$:
\begin{equation}\label{mdiag}
\xymatrix{
& M\left(C_1H\ot V ,\, C_2H\ot V,  \chi\right) 
\ar[dl]_{\mathrm{pr}_1} \ar[dr]^{\mathrm{pr}_2}& \\
C_1H\ot V \ar[d]_{\mathrm{ev}_{\frac{1}{2}}\ot\id}
%{\pi_{1\ast}} 
&    &  C_2H\ot V \ar[d]^{\mathrm{ev}_{\frac{1}{2}}\ot\id} \\
%{\pi_{2\ast}}\\
H\ot V\ar[rr]_{\chi} & &H\ot V\,. }
\end{equation}

\section{Noncommutative Borsuk-Ulam framework}

\subsection{Borsuk-Ulam theorem for quantum spheres}

Applying the Gelfand transform, Theorem~\ref{classicalequi} translates to:
\[\label{classphere}
\text{\large\boldmath$\not\!\exists$}\;\;\text{$\z2$-equivariant $*$-homomorphism}\; 
C(S^n)\longrightarrow C(S^{n+1}) \, .
\]
Replacing the commutative C*-algebras of functions on spheres by noncommutative C*-algebras of 
$q$-deformed spheres,
we obtain a noncommutative version of the Borsuk-Ulam Theorem.

In particular, we can consider it for the case of the equatorial Podle\'s quantum two-sphere~\cite{p-p87}.
We tensor the C*-algebra $C(S^2_{q\infty})$ of the equatorial Podle\'s quantum sphere 
with the algebra of continuous 
functions on the unit circle, act on the tensor product with the diagonal antipodal 
$\mathbb{Z}/2\mathbb{Z}$-action, and consider the invariant subalgebra. This gives a 
$U(1)$-C*-algebra $A$ with the quantum real projective space C*-algebra $C(RP_q(2))$
(see~\cite{hms03}) as its 
$U(1)$-invariant part. Using the identity representation of $U(1)$, we associate with it a 
finitely generated projective module over $C(RP_q(2))$. 
One can prove that this module is not stably free~\cite{bhr}. 
This implies that $A$ cannot be 
a crossed product of $C(RP_q(2)$) and the integers. This proves the quantum 
Borsuk-Ulam Theorem for~$n=1$.
For $n=1$ and $q=1$, this is a proof of the weather-on-Earth case of the Borsuk-Ulam Theorem 
(see the introduction). 

For the arbitrary-dimension quantum spheres introduced in \cite{vs90,hs03}, 
the quantum Borsuk-Ulam Theorem
\[
\text{\large\boldmath$\not\!\exists$}\;\;\text{$C(\mathbb{Z}/2\mathbb{Z})$-equivariant 
$*$-homomorphism}\; C(S^n_q)\longrightarrow 
C(S^{n+1}_q)
\]
is proven in~\cite[Theorem 3]{y-m13}.
Let us emphasize that our
 noncommutative generalization of \eqref{classphere} is different from the one proved in~\cite{y-m13}.

\subsection{Noncommutative Borsuk-Ulam-type conjectures}

The Borsuk-Ulam Theorem is equivalent to:
\begin{theorem}[join formulation]
Let $n$ be a positive natural number.
There  does \emph{not} exist a $\mathbb{Z}/2\mathbb{Z}$-equivariant 
continuous  map $\widetilde{f}\colon S^{n-1}*\mathbb{Z}/2\mathbb{Z}\to S^{n-1}$. 
\end{theorem}

This naturally leads to a classical Borsuk-Ulam-type conjecture:
\begin{conjecture}\label{mainclass}
Let $X$ be a compact Hausdorff space equipped with a continuous free action of a 
non-trivial compact Hausdorff
group~$G$. Then, for the diagonal action of $G$ on $X*G$, there does \emph{not} exist a $G$-equivariant
continuous map $f:X*G\to X$.
\end{conjecture}
\noindent
Particular cases of the above statement going beyond the Borsuk-Ulam Theorem 
have already been studied in \cite{s-e40,g-h37,g-h43}
(cf.~\cite{t-a12} for weaker results for non-free  
$\mathbb{Z}/p\mathbb{Z}$-actions and maps from $X$ to~$S^1$).

Thus we have arrived at the main point of this paper:
\begin{conjecture}\label{main}
Let $A$ be a unital C*-algebra with a free action of a non-trivial compact quantum group~$(H,\Delta)$. 
Also, let $A*_\delta H$ be the equivariant noncommutative join C*-algebra of $A$ and $H$ with the induced
free action of $(H,\Delta)$ given by $\delta_\Delta$. Then
\begin{align*}
&\text{\sc type 1}\qquad\boxed{\overset{\phantom\delta}{\phantom *}
\text{\large\boldmath$\not\!\exists$}\;\;\text{$H$-equivariant $*$-homomorphism}\; A\longrightarrow 
A\underset{\delta}{*}H\overset{\phantom\delta}{\phantom *}}\;,\\
&\text{\sc type 2}\qquad\boxed{\overset{\phantom\delta}{\phantom *}
\text{\large\boldmath$\not\!\exists$}\;\;\text{$H$-equivariant $*$-homomorphism}\; H\longrightarrow 
A\underset{\delta}{*}H\overset{\phantom\delta}{\phantom *}}\;.\vspace*{-5mm}\\
\end{align*}

\vspace*{-8mm}\noindent
Here the $H$-equivariance is defined with respect to coactions $\delta$ and $\delta_\Delta$ (type 1), and 
with respect to $\Delta$ and $\delta_\Delta$ (type 2).
\end{conjecture}

\section{Noncontractibility of compact quantum groups}

\noindent
In this section, we consider the special case $(A,\delta)=(H,\Delta)$ when
the two types of Conjecture~\ref{main} coincide. In the classical setting,
the thus restricted Conjecture~\ref{main} boils down to the non-contractibility of non-trivial compact
Hausdorff groups, which is well known~\cite{h-b79}.

\subsection{General setting}

If $X$ is a compact Hausdorff principal $G$-bundle, $A=C(X)$ and $H=C(G)$,  then 
Conjecture~\ref{main} type 2 states that the principal $G$-bundle $X*G$ is not
trivializable unless $G$ is trivial. This is clearly true because otherwise $G*G$ would be trivializable, which
is tantamount to $G$ being contractible, and the only compact Hausdorff contractible group is trivial. 
 On the other hand, we do not know whether Conjecture~\ref{main} type 1 holds in general
in the commutative case (see Conjecture~\ref{mainclass}).

\begin{definition}\label{triv}
Let $A$ be a unital C*-algebra equipped with a free action of a compact quantum group $(H,\Delta)$
implemented by $\delta\colon A\to A\otimes_\mathrm{min}H$, and
let $B$ denote the fixed-point subalgebra for this action. We call such a triple $(P, B, H)$ a \emph{compact
quantum principal bundle}. We say that
a compact quantum principal bundle $(P, B, H)$ is {\em trivializable} iff
there exists an $H$-equivariant $*$-homomorphism $j:H\to P$,
i.e.\ a $*$-homomorphism $j$ such that $(j\ot\id)\circ\Delta=\delta\circ j$.
\end{definition}

\begin{proposition}\label{prop}
Let $(H,\Delta)$ be a compact quantum group, $\mathcal{O}(H)$ be its 
dense Hopf $*$-subalgebra
spanned by the matrix coefficients of irreducible unitary corepresentations.
If $(P, B, H)$ is a trivializable compact quantum principal bundle
then, for any finite-dimensional corepresentation $\varrho : V \to \mathcal{O}(H)\otimes V$,
the associated left $B$-module $\mathcal{P}_H(P) \Box_{\mathcal{O}(H)} V$ is free. 
\end{proposition}
\begin{proof}
Due to the $H$-colinearity of $j$, for any $h\in H$, we obtain
$
j(h_{(1)})\otimes h_{(2)} = \delta ( j(h)).
$
This shows that 
$
j(\mathcal{O}(H)) \subseteq \mathcal{P}_H(P).
$
Therefore, $\mathcal{P}_H(P)$ is a smash-product comodule algebra.
Hence the associated left $B$-module 
$ \mathcal{P}_H(P) \Box_{\mathcal{O}(H)} V$ is free.
\end{proof}

\begin{theorem}\label{mainthm}
Let $(H,\Delta)$ be a compact quantum group, $\mathcal{O}(H)$ be its 
dense Hopf $*$-subalgebra
spanned by the matrix coefficients of irreducible unitary corepresentations,
and \mbox{$V\stackrel{\varrho}{\to} \mathcal{O}(H)\otimes V$} be
 a finite-dimensional corepresentation of $\mathcal{O}(H)$.
Next, let $\mathcal{P}_H(H\!*_\Delta\! H)$ be the Peter-Weyl comodule algebra of
the equivariant noncommutative join $H\!*_\Delta\! H$, let
$\Sigma H$ be the unreduced suspension C*-algebra of~$H$
(the coaction-invariant subalgebra $\mathcal{P}_H(H\!*_\Delta\! H)^{co\,\mathcal{O}(H)}$), 
and let $C_1H$ and $C_2H$ be
respectively the left and right cones of~$H$.
Then 
\[\label{mod}
\mathcal{P}_H(H\underset{\Delta}{*} H) \underset{\mathcal{O}(H)}{\Box} V\cong
M\big(C_1H\otimes V\,,\,C_2H\otimes V\,,\,(m\otimes\id)\circ(\id\otimes(S\otimes\id)\circ\varrho)\big)
\]
as left modules over~$\Sigma H$. Here $m$ and $S$ are respectively the multiplication and the antipode
of the Hopf algebra~$\mathcal{O}(H)$, and the right-hand-side module is defined as the pullback
module of~\eqref{mdiag}.
\end{theorem}
\begin{proof}
Consider $C_iH$-module isomorphisms $\Lambda_i$, $i=1,2$, given by
\begin{gather}
\Lambda_i \colon C_iH\otimes V
\stackrel{\id\otimes\varrho}{\longrightarrow}
C_iH\otimes \mathcal{O}(H)\underset{\mathcal{O}(H)}{\Box}V
\stackrel{J_i\otimes\id}{\longrightarrow}
\mathcal{P}_i(H\underset{\Delta}{*} H)\underset{\mathcal{O}(H)}{\Box} V,\\
J_i\colon C_iH\otimes \mathcal{O}(H)\ni c_i\otimes h\longmapsto c_ij_i(h)\in
 \mathcal{P}_i(H\underset{\Delta}{*} H),
\nonumber\\
j_1\colon \mathcal{O}(H)\ni  h\longmapsto 1\otimes h\in\mathcal{P}_1(H\underset{\Delta}{*} H),
\nonumber\\
j_2\colon \mathcal{O}(H)\ni  h\longmapsto 
(t\mapsto h_{(1)})\otimes h_{(2)}\in\mathcal{P}_2(H\underset{\Delta}{*} H).
\nonumber
\end{gather}
Since the modules of \eqref{mod} have the same pullback structure (see Section~1.3), they are isomorphic
if the following equivalence holds:
\begin{gather}
\left((m\otimes\id)\circ(\id\otimes(S\otimes\id)\circ\varrho)\circ
(\mathrm{ev}_{\frac{1}{2}}\otimes\id)\right)(b_1)
=\left(\mathrm{ev}_{\frac{1}{2}}\otimes\id\right)(b_2)
\label{X1}
%\nonumber
\\
\big\Updownarrow\nonumber\\
\left((\mathrm{ev}_{\frac{1}{2}}\otimes\id)\circ\Lambda_1 \right)(b_1)
=
\left((\mathrm{ev}_{\frac{1}{2}}\otimes\id)\circ\Lambda_2 \right)(b_2).
\label{X2}
%\nonumber
\end{gather}
On simple tensors $b_i=c_i\ot v_i$, the formula \eqref{X1} boils down to
\[\label{X11}
c_1(\half) S( v_{1(-1)})\ot v_{1(0)}
=c_2(\half) \ot v_2\, ,
\]
and the formula \eqref{X2} becomes
\[\label{X12}
c_1(\half)  v_{1(-1)}\ot v_{1(0)}
=c_2(\half) v_{2(-2)}  v_{2(-1)} \ot v_{1(0)}.
\]
Applying twice the isomorphism $(m\ot\id )\circ (\id\ot\varrho)$ (its inverse is
$(m\ot\id )\circ (\id\ot(S\ot\id)\circ\varrho)$
to both sides of
\eqref{X11} yields \eqref{X12}.
\end{proof}

\subsection{Quantum instanton bundle}

Assume now that
$H$ is the C*-algebra $C(SU_q(2))$ of the quantum group $SU_q(2)$. 
Then the equivariant noncommutative join $C(SU_q(2))*_\Delta C(SU_q(2))$ is Pflaum's quantum instanton
fibration~\cite{p-mj94,hkmz11}. 
Let $\alpha$ and $\beta$ 
be the usual generators of~$C(SU_q(2))$.
Also, let  
\begin{equation}
\label{K1gen}
U:= \left[ 
\begin{array}{cc}
\alpha & -q \beta^\ast \cr
\beta & \alpha^\ast 
\end{array}
\right] \in U_2 (C(SU_q(2))) \subset M_{2}(C(SU_q(2))), 
\end{equation}
and let $V=\C^2$ be a left $\mathcal{O}(H)$-comodule via
\[\label{rho}
\varrho_U: \C^2\lra \mathcal{O}(H)\ot \C^2,\qquad \varrho_U(e_i) := \sum_{j=1}^2U_{ij}\otimes e_j\,,
\] 
where $\{e_i\}_i$ is the standard basis of~$\mathbb{C}^2$.
(This corepresentation of $C(SU_q(2))$
is usually called the \emph{fundamental representation}
of~$SU_q(2)$.) It follows from \eqref{mod} that the left $\Sigma C(SU_q(2))$-module
\[
E:=\mathcal{P}_{C(SU_q(2))}\Big(C(SU_q(2))\underset{\Delta}{*} C(SU_q(2))\Big) 
\underset{\mathcal{O}(C(SU_q(2)))}{\Box}\mathbb{C}^2
\]
is isomorphic to the pullback module \eqref{mdiag} for
$\chi= (m\otimes\id)\circ(\id\otimes(S\otimes\id)\circ\varrho_U)$.
Hence, by \cite[Theorem~2.1]{dhhmw12}, an idempotent $p$ representing the $K_0$-class of $E$
can be computed via \cite[(2.7)]{dhhmw12} using the unitary matrix $a=U^*$.
Moreover, as explained in
\cite[p.~77]{dhhmw12}, there exists an even Fredholm module whose index pairing with $E$ 
is equal to the index pairing of $U^*$ with an appropriate odd Fredholm module. 
As the latter equals $-1\neq0$, we infer that the module $E$ is not free. Therefore, Proposition~\ref{prop}
implies that  Pflaum's quantum instanton
fibration is not trivializable, i.e.~Conjecture~\ref{main} holds. As explained at the beginning of Section~3.1,
this means that $SU_q(2)$ is \emph{not} contractible.

\section*{Acknowledgements}
\noindent
This work was partially supported by NCN grant 2011/01/B/ST1/06474. The authors are very grateful to the
Hausdorff Research Institute for Mathematics in Bonn, where key progress on the paper was made, for its
fabulous hospitality and support. It is a pleasure to thank Kenny De Commer and Makoto Yamashita for their
help and lengthy discussions.
Paul F.\ Baum and Ludwik D\k{a}browski   were  also partially supported 
by  NSF grant DMS 0701184 and PRIN 2010-11 grant ``Operator Algebras,
 Noncommutative Geometry and Applications", respectively.

\end{document}